\documentclass[12pt, reqno,draft]{amsart}

\usepackage{amsmath}
\usepackage{amsthm}
\usepackage{amssymb}
\usepackage{amsrefs}
\usepackage{mathrsfs}
\usepackage[usenames]{color}

 \newtheorem{thm}{}[section]
 \newtheorem{theorem}[thm]{Theorem}
 
 \newtheorem{lemma}[thm]{Lemma}

  \theoremstyle{definition}
 
 \theoremstyle{remark}

 \numberwithin{equation}{section}
 \allowdisplaybreaks

\newcommand{\NN}{\ensuremath{\mathbb{N}}}
\newcommand{\ee}{\ensuremath{\mathbf{e}}}
\newcommand{\ww}{\ensuremath{\mathbf{w}}}
\newcommand{\WW}{\ensuremath{\mathcal{W}}}
\newcommand{\OO}{\ensuremath{\mathcal{O}}}

\newcommand{\DD}{\ensuremath{\mathcal{D}}}

\begin{document}

\title[Optimality of the rearrangement inequality]{Optimality of the rearrangement inequality with applications to Lorentz-type sequence spaces}

\author[F. Albiac]{Fernando Albiac}
\address{Mathematics Department\\ 
Universidad P\'ublica de Navarra\\
Campus de Arrosad\'{i}a\\
Pamplona\\ 
31006 Spain}
\email{fernando.albiac@unavarra.es}

\author[J. L. Ansorena]{Jos\'e L. Ansorena}
\address{Department of Mathematics and Computer Sciences\\
Universidad de La Rioja\\ 
Logro\~no\\
26004 Spain}
\email{joseluis.ansorena@unirioja.es}

\author[D. Leung]{Denny Leung}
\address{ Department of Mathematics \\ National University of Singapore \\ Singapore 117543
Republic of Singapore}
\email{matlhh@nus.edu.sg}

\author[B. Wallis]{Ben Wallis}
\address{Department of Mathematical Sciences\\ Northern Illinois University\\ DeKalb, IL 60115 U.S.A.}
\email{benwallis@live.com}

\begin{abstract} We characterize the sequences $(w_i)_{i=1}^\infty$ of non-negative numbers for which
\[
\sum_{i=1}^\infty a_i w_i  
\quad
\text{ is of the same order as }
\quad
\sup_n \sum_{i=1}^n a_i w_{1+n-i}
\]
when 
$(a_i)_{i=1}^\infty$ runs over all non-increasing sequences of non-negative numbers. As a by-product of our work we settle a problem raised in \cite{AAW} and prove that Garling sequences spaces have no symmetric basis.
\end{abstract}



\thanks{The research of the two firs-named authors was partially supported by the Spanish Research Grant \textit{An\'alisis Vectorial, Multilineal y Aplicaciones}, reference number MTM2014-53009-P. F. Albiac also acknowledges the support of Spanish Research Grant \textit{Operators, lattices, and structure of Banach spaces},  with reference MTM2016-76808-P. The research of the third author was partially supported by AcRF grant R-146-000-242-114.
}

\maketitle

\section{Introduction}\label{Sec1}
\noindent The rearrangement inequality  states that, for $n\in \NN$, if  $(a_i)_{i=1}^n$ and  $(b_i)_{i=1}^n$ are
 a pair of non-increasing $n$-tuples of non-negative scalars then we have 
\[
\sum_{i=1}^n a_i b_{1+n-i} \le \sum_{i=1}^n a_i b_{\sigma(i)}\le  \sum_{i=1}^n a_i b_{i}
\]
for every permutation $\sigma$ of the set $\{1,\dots,n\}$ (see \cite{HLP}*{Theorem 368}). Consequently, if  $(a_i)_{i=1}^\infty$ and  $(w_i)_{i=1}^\infty$ 
are non-increasing sequences of  non-negative scalars,
\[
\sup_{n\in\NN} \sum_{i=1}^n a_i w_{1+n-i} \le \sum_{i=1}^\infty a_i w_i.
\]
In this note we wonder about which are the non-increasing sequences $(w_i)_{i=1}^\infty$ of non-negative scalars that verify a reverse inequality, i.e., in which cases   there  is a constant $C<\infty$ such that
\begin{equation}\label{revineq}
 \sum_{i=1}^\infty a_i w_i \le C \sup_{n\in\NN} \sum_{i=1}^n a_i w_{1+n-i}
\end{equation}
for every sequence $(a_i)_{i=1}^\infty$ of non-negative scalars.
For the time being some simple answers can be given.  Indeed, on the one hand, if $w_\infty:=\inf_i w_i>0$ then
\[
\sum_{i=1}^\infty a_i w_i 
\le w_1\sum_{i=1}^\infty a_i=w_1\sup_n \sum_{i=1}^n a_i
 \le \frac{w_1}{w_\infty} \sup_n \sum_{i=1}^n a_i w_{1+n-i}.
\]
 On the other hand, if we consider $W:=\sum_{i=1}^\infty w_i<\infty$ and let $w_1>0$ (the case $w_1=0$ is trivial) then
 \[
\sum_{i=1}^\infty a_i w_i 
\le a_1 \sum_{i=1}^\infty w_i   = \frac{W}{w_1} a_1 w_1
 \le \frac{ W}{w_1}  \sup_n \sum_{i=1}^n a_i w_{1+n-i}.
\]
In fact, as we will show below, these two   cases are the only ones for which \eqref{revineq}
holds. This will be our main result as far as inequalities is concerned: 
\begin{theorem}[Main Theorem]\label{MainTheorem}Let $(w_i)_{i=1}^\infty$ be a non-increasing sequence consisting of non-negative scalars. The following are equivalent:
\begin{itemize}
\item[(i)] There is a constant 
$C<\infty$ such that
\[
 \sum_{i=1}^\infty a_i w_i \le C \sup_{n\in\NN} \sum_{i=1}^n a_i w_{1+n-i}
 \]
 for every sequence $(a_i)_{i=1}^\infty$ of non-negative scalars.

\item[(ii)]  Either $\sum_{i=1}^\infty w_i<\infty$ or $\inf_{i\in\NN} w_i>0$.
\end{itemize}
\end{theorem}

Section~\ref{SectMain} is devoted to proving Theorem~\ref{MainTheorem}.  In Section~\ref{SectGarling} we use Theorem~\ref{MainTheorem} to give some functional analytic properties  of a recently introduced class of Lorentz-type spaces, called Garling sequence spaces.  In particular, Theorem~\ref{MainTheorem} is applied to show that Garling sequence spaces have no symmetric basis, answering thus a problem that was recently posed in \cite{AAW}.

Throughout this note we use standard terminology and notation in Banach space theory. As it is customary,  we denote by $\ell_q$, $1\le q\le \infty$, the Banach space consisting of all $q$-summable sequences
(bounded sequences in the case $q=\infty$) and by
 $c_0$ the subspace of $\ell_\infty$ consisting of all sequences converging to zero. 
For  background on bases in Banach spaces we refer the reader to   \cite{AK2016}.

\section{Proof of the Main Theorem}\label{SectMain}

\begin{proof}[Proof of Theorem~\ref{MainTheorem}] As explained in the  Introduction, we  must only prove that
(i) implies (ii). 

Assume that (ii) does not hold, that is, $\ww=(w_i)_{i=1}^\infty\in c_0 \setminus \ell_1$.
Let us denote by $\DD$ the set of (nonzero) non-increasing sequences of non-negative integers. For $f=(a_i)_{i=1}^\infty\in\DD$  and $n\in\NN$ we put
\begin{align*}
A(f,\ww)&=\sum_{i=1}^\infty a_i w_i \text{, and}\\
B(f,\ww)&=\sup_{n\in\NN} B(f,\ww,n),
\end{align*}
where, for $n\in\NN$,
\[
B(f,\ww,n)= \sum_{i=1}^n a_i w_{1+n-i}.
\]
With this notation we must prove that
\[
S(\ww):=\sup_{f\in\DD} \frac{A(f,\ww)}{ B(f,\ww)}=\infty.
\]
We will use the convention that $\sum_{i=1}^0 c_i=0$ for all sequences of scalars $(c_i)_{i=1}^\infty$.

For $n\in \NN$ put $W(n) = \sum^n_{i=1}w_i$.
Since $\ww\notin\ell_1$ we have
\[
\lim_n W(n)=\infty.
\]
Moreover, since $\ww\in c_0$,
\[
\lim_{n\in\NN} (W(s+n) - W(n))=0
\]
for any non-negative integer $s$. We use these properties to recursively construct an increasing sequence
$(d_k)_{k=0}^\infty$ of non-negative integers with $d_0=0$ verifying  
\begin{enumerate}
\item[(i)] $W(\sum^{k-1}_{j=1}d_j) \leq  2^{-1} W(d_k)$, and
\item[(ii)] $W(d_{k-1}+d_k) -W(d_k)\leq 2^{1-k} W(d_{k-1})$
\end{enumerate}
for $k=1,2,\dots$.

For every integer  $k\ge 0$ put $n_k=\sum^{k}_{j=1}d_j$. 
For each $r\in\NN$ we define a sequence $f^{(r)}=(a_{i,r})_{i=1}^\infty$ by 
\[
a_{i,r}=
\begin{cases}  {1}/{W(d_k)} & \text{ if, for some $1\le k\le r$, }  n_{k-1} <i \le  n_k\\
0 &\text{ if } i> n_r.
\end{cases}
\]
It is clear that $f^{(r)}\in\DD$ for all  $r\in\NN$. Taking into account the inequality in (i) we obtain
\begin{align*}
A(f^{(r)},\ww)&=  \sum_{k=1}^r \frac{1}{W(d_k)} \sum_{i=1+n_{k-1}}^{n_k} w_i \\
&=  \sum^r_{k=1}\frac{W(n_k)-W(n_{k-1})}{W(d_k)}\\
& \geq \sum^r_{k=1} \frac{W(d_k)-2^{-1} W(d_k)}{W(d_k)}\\
& = \sum^r_{k=1}\frac{1}{2} \\
&=\frac{r}{2}.
\end{align*}
Let $n\in\NN$. In  case that $n>n_r$ we have
\[
B(f^{(r)},\ww,n)=\sum_{i=1}^{n_r} a_{i,r} w_{1+n-i}\le \sum_{i=1}^{n_r} a_{i,r} w_{1+n_r-i}=B(f^{(r)},\ww,n_r).
\]
In  case that $n\le n_r$, pick $1\le k \le r$ with $n_{k-1}<n\le n_k$. We have
\begin{align*}
B(f^{(r)},\ww,n)&=\frac{W(n-n_{k-1})}{W(d_k)}+\sum_{j=1}^{k-1} \frac{W(n-n_{j-1})-W(n-n_{j})}{W(d_j)}\\
&\le \frac{W(n_k-n_{k-1})}{W(d_k)}+\sum_{j=1}^{k-1} \frac{W(n-n_{j-1})-W(n-n_{j})}{W(d_j)}\\
&=1+\sum_{j=1}^{k-1} \frac{W(n-n_{j-1})-W(n-n_{j})}{W(d_j)}.
\end{align*}
If $k=1$ we get $B(f^{(r)},\ww,n)\le 1$.
Assume that $k\ge 2$. Taking into account inequality~(ii) and that, since $\ww$ is non-increasing, the sequence $(W(n+t)-W(n+s))_{n=1}^\infty$ is non-increasing for any $s\le t$, we obtain

\begin{align*}
B(f^{(r)},\ww,n)
&\le 1+ \frac{W(n-n_{k-2})-W(n-n_{k-1})}{W(d_{k-1})}\\
&\qquad+\sum_{j=1}^{k-2} \frac{W(n-n_{j-1})-W(n-n_{j})}{W(d_j)}\\
&\le 1+\frac{W(n_{k-1}-n_{k-2})-W(n_{k-1}-n_{k-1})}{W(d_{k-1})}\\
&\qquad+\sum_{j=1}^{k-2} \frac{W(n_{j+1}-n_{j-1})-W(n_{j+1}-n_{j})}{W(d_j)}\\
&= 2+\sum_{j=1}^{k-2} \frac{W(d_{j+1}+d_j)-W(d_{j+1})}{W(d_j)}\\
&\le 2+\sum_{j=1}^{k-2} 2^{-j}\\
&=3-2^{2-k}.
\end{align*}
Therefore $B(f^{(r)}, \ww)\le 3$. Thus
\[
S(\ww)\ge \sup_{r\in\NN} \frac{A(f^{(r)}, \ww)}{B(f^{(r)}, \ww)}\ge \sup_{r\in\NN} \frac{r}{6}=\infty,
\]
and the proof is over.
\end{proof}

\section{Applications to Garling squence spaces}\label{SectGarling}
\noindent Let  $1\le p<\infty$ and let $\ww=(w_n)_{n=1}^\infty$ be a non-increasing sequence of positive scalars.  Given a sequence  of (real or comnplex) scalars $f=(b_k)_{k=1}^\infty$ we put
\begin{equation*}\label{GarlingNorm}
 \Vert  f  \Vert_{g(\ww,p)} = \sup_{\phi\in\OO } \left( \sum_{i=1}^\infty |b_{\phi(i)}|^p w_i \right)^{1/p}
 \end{equation*}
where $\OO$ denotes the set of increasing functions from $\NN$ to $\NN$. 
The Garling sequence space $g(\ww,p)$ is the Banach space consisting of all sequences $f$ with  
$ \Vert  f  \Vert_{g(\ww,p)}<\infty$.

Notice that if in  \eqref{GarlingNorm} we replace ``$\phi\in\OO$'' with ``$\phi$ is a permutation of $\NN$'' 
we obtain the norm that defines the weighted Lorentz sequence space 
\[
d(\ww,p):=\left\{ (b_k)_{k=1}^\infty \in c_0 \colon   \left( \sum_{i=1}^\infty  (b_i^*)^p w_i \right)^{1/p}<\infty\right\},
\] 
where
 $(b_i^*)_{i=1}^\infty$  denotes the decreasing rearrangement of $(b_k)_{k=1}^\infty$. So, the Garling sequence space $g(\ww,p)$ can be regarded as a variation of the weighted Lorentz sequence space $d(\ww,p)$.

Imposing the further conditions $\ww\in c_0$ and $\ww\notin\ell_1$ will prevent us, respectively, from having $g(\ww,p)=\ell_p$ or $g(\ww,p)= \ell_\infty$. We will assume as well  that $\ww$ is normalized, i.e., $w_1=1$.
Thus, we put
\[\WW:=\left\{(w_i)_{i=1}^\infty\in c_0\setminus\ell_1:1=w_1\geq w_2\ge\cdots \ge w_i \ge w_{i+1} \ge\cdots>0\right\}\]
and we restrict our attention to \textit{weights} $\ww\in\WW$.

For $n\in\NN$, we will denote $\ee_n=(\delta_{i,n})_{i=1}^\infty$, where $\delta_{i,n}=1$ if $n=i$ and $\delta_{i,n}=0$ otherwise. We have (see \cite{AAW}*{Theorem 3.1}) that the canonical sequence $(\ee_n)_{n=1}^\infty$ is a Schauder basis for $g(\ww,p)$. A question posed and partially solved in \cite{AAW} is to determine the weights $\ww\in\WW$ and the indices $p\in[1,\infty)$ for which $(\ee_n)_{n=1}^\infty$ is a symmetric basis of $g(\ww,p)$. 
Here we provide a complete intrinsic solution to this problem, in the sense that our approach is entirely based on Theorem~\ref{MainTheorem}.
\begin{lemma}\label{NewCharacterization}
The canonical sequence
$(\ee_n)_{n=1}^\infty$ is not a symmetric basis for $g(\ww,p)$ for any $\ww\in\WW$ and any $1\le p<\infty$.
\end{lemma}

\begin{proof}Assume that $(\ee_n)_{n=1}^\infty$ is a symmetric basis for $g(\ww,p)$. Then, there is a constant $C$ so that
\[
\Vert g \Vert_{g(\ww,p)}\le C \Vert f\Vert_{g(\ww,p)}
\]
whenever the sequence $g$ is a permutation of the sequence $f$.

Given $r\in\NN$ and $\phi\in\OO$ let  $n(r,\phi)$ be the largest integer $n$
such that $\phi(n)\le r$. We have   $\phi(i)\le i+r-n(r,\phi)$ for $1\le i\le n(r,\phi)$. 
Given a non-increasing sequence $(a_i)_{i=1}^\infty$ of non-negative numbers we have
\begin{align*}
\sum_{i=1}^\infty a_i w_i
&=\sup_r \sum_{i=1}^r a_i w_i\\
&\le \sup_{r\in\NN} \left\Vert \sum_{i=1}^r a_i^{1/p} \ee_i \right\Vert_{g(\ww,p)}^p\\
&\le C \sup_{r\in\NN}  \left\Vert\sum_{i=1}^r a_{1+r-i}^{1/p} \ee_i \right\Vert_{g(\ww,p)}^p\\
& =C\sup_{r\in\NN,\,\phi\in\OO} \sum_{i=1}^{n(r,\phi)}  a_{1+r-\phi(i)}  w_i\\
& \le C\sup_{r\in\NN,\,\phi\in\OO} \sum_{i=1}^{n(r,\phi)}  a_{1+n(r,\phi)-i}  w_i\\
&=C \sup_{n\in\NN}   \sum_{i=1}^{n}  a_{1+n-i}  w_i\\
&=C\sup_{n\in\NN}  \sum_{i=1}^{n}  a_i w_{1+n-i}.
\end{align*}
 Theorem~\ref{MainTheorem} yields the absurdity $\ww\in\ell_1$ or $\ww\notin c_0$.
\end{proof} 
Now we are ready to establish the advertised   structural  properties  of Garling sequence spaces.
\begin{theorem}Let $\ww\in\WW$ and $1\le p<\infty$.
\begin{enumerate}

\item[(i)] There is no symmetric basis for $g(\ww,p)$.

\item[(ii)]    $d(\ww,p)\subsetneq g(\ww,p)$.

\item[(iii)]  No subspace of $d(\ww,p)$ is isomorphic to $g(\ww,p)$.

\item[(iv)]  Let  $I_{d,g}\colon  d(\ww,p) \to g(\ww,p)$ be  the natural inclusion map, and let
$T\colon g(\ww,p) \to d(\ww,p)$ be a bounded linear operator. Then (despite the fact that $I_{d,g}$ is not a strictly singular operator) $T\circ I_{d,g}$  does not preserve a copy of $d(\ww,p)$, i.e., if $X$ is a subspace of $d(\ww,p)$ isomorphic to $d(\ww,p)$ then $T\circ I_{d,g}|_X$ is not an isomorphism.
 \end{enumerate}
\end{theorem}

\begin{proof}It follows using   Lemma~\ref{NewCharacterization} in combination with \cite{AAW}*{Theorem 5.1}.
\end{proof}



\begin{bibsection}
\begin{biblist}

 \bib{AAW}{article}{
 author={Albiac, F.},
 author={Ansorena, J.L.},
   author={Wallis, B.},
   title={On Garling sequence spaces},
   journal={arXiv:1703.07772 [math.FA]},
}

\bib{AK2016}{book}{
   author={Albiac, F.},
   author={Kalton, N.J.},
   title={Topics in Banach space theory},
   series={Graduate Texts in Mathematics},
   volume={233},
   edition={2},
   publisher={Springer, [Cham]},
   date={2016},
   pages={xx+508},
}

\bib{HLP}{book}{
   author={Hardy, G. H.},
   author={Littlewood, J. E.},
   author={P\'olya, G.},
   title={Inequalities},
   series={Cambridge Mathematical Library},
   note={Reprint of the 1952 edition},
   publisher={Cambridge University Press, Cambridge},
   date={1988},
   pages={xii+324},
}
\end{biblist}
\end{bibsection}
\end{document}